\documentclass[10pt, reqno]{amsart}
\usepackage{amsmath,amsfonts,amsthm,bm} 

\usepackage{amsfonts}
\usepackage{amssymb}
\usepackage{latexsym}

\usepackage{tikz}
\usepackage{graphicx,graphics,epsfig}
\usepackage{epstopdf}

\usepackage[]{enumerate}
\usepackage{verbatim}
\usepackage{bbm}

\usepackage[]{enumerate}

\usepackage[sc]{mathpazo}          
\usepackage{eulervm}               
\usepackage[scaled=0.86]{berasans} 
\usepackage[scaled=1]{inconsolata} 
\usepackage[T1]{fontenc}

\usepackage[final]{hyperref}
\hypersetup{colorlinks=true, linkcolor=blue, anchorcolor=blue, citecolor=red, filecolor=blue, menucolor=blue, pagecolor=blue, urlcolor=blue}

\newcommand{\etal}{\textit{et al.}}

\newcommand{\Bigabs}[1]{\Bigl\vert #1 \Bigr\vert}

\newcommand{\norm}[1]{\left\Vert #1 \right\Vert}

\newcommand{\Z}{\mathbb{Z}}
\newcommand{\R}{\mathbb{R}}

\newcommand{\angles}[1]{\langle #1 \rangle}

\DeclareMathOperator{\supp}{supp}

\newtheorem{theorem}{Theorem}

\newtheorem{lemma}{Lemma}

\theoremstyle{definition}

\theoremstyle{remark}
\newtheorem{remark}{Remark}

\title{Long-time existence for a Whitham--Boussinesq system
	in two dimensions}

 \author[A. Tesfahun]{Achenef Tesfahun}

\address{Department of Mathematics \\
Nazarbayev University \\
Qabanbai Batyr Avenue 53 \\
010000 Nur-Sultan \\
Republic of Kazakhstan}

\email{achenef@gmail.com}

 \keywords{ Surface waves, Witham-Boussinesq systems, Long-time existence, dispersive estimates}

\subjclass[2010]{5Q53, 35Q35, 76B15, 35A01, 76B03}

\begin{document}

\begin{abstract} 
This paper is concerned with a two dimensional Whitham–Boussinesq system modeling surface waves of an inviscid incompressible fluid layer. We prove that the associated Cauchy problem is well-posed for initial data of low regularity, with existence time of scale $\mathcal O\left(  \mu^{3/2-} \epsilon^{-2+}\right)$, where $\mu$ and $\epsilon$ are small parameters related to the level of dispersion and nonlinearity, respectively. In particular, in the KdV regime $\{\mu \sim \epsilon $\}, the existence time is of order $ \epsilon^{-1/2}$.
The main ingredients in the proof are frequency loacalised dispersive estimates and bilinear Strichartz estimates that depend on the parameter $\mu$.

\end{abstract}

\maketitle
\section{Introduction}

We consider the Cauchy problem for Whitham--Boussinesq system

\begin{equation}
\label{wt2d}
\left\{
\begin{aligned}
  \partial_t \eta +  \nabla \cdot  \mathbf v & = -\epsilon K_\mu\nabla \cdot (\eta \mathbf v),
  \\
  \partial_t \mathbf v + K_\mu\nabla \eta  &=   - \epsilon  K_\mu \nabla ( |\mathbf v |^2/2),
  \\
  (\eta, \mathbf v)(0)&=( \eta_0, \mathbf v_0),
\end{aligned}
\right.
\end{equation} 
where $\eta: \R^{2+1}\mapsto \R$, $ \mathbf v:  \R^{2+1}\mapsto  \R^2$ is a curl--free vector field, i.e., $
 	\nabla \times  \mathbf v   = 0$, and 
\begin{equation*}
	K_\mu:=K_\mu(D)= \frac{\tanh( \sqrt \mu|D|)}{\sqrt \mu |D|}\quad \text{with}  \ D=-i\nabla.
\end{equation*}

The system \eqref{wt2d}
describes the evolution with time of surface waves of
a liquid layer in the three dimensional physical space. The 
 variables $\eta$ and $\mathbf v$
denote the surface elevation
and the fluid velocity, respectively. 
The shallowness parameter $\mu$ and the nonlinearity parameter $\epsilon$ are defined by 
 $$\mu=( h/\lambda)^2, \qquad \epsilon=a /h,$$
 where $h$ denotes the mean depth of the fluid layer,  $a$ is a typical amplitude of the wave, $\lambda$ a typical horizontal wavelength.
 
  For $\mu \ll 1 $, one has formally 
$$
K_\mu(D)= 1 +\frac{\mu}{3} \Delta + \mathcal O (\mu^2),
$$
and hence one can write \eqref{wt2d}
as
\begin{equation*}
\left\{
\begin{aligned}
\partial_t \eta +  \nabla \cdot  \mathbf v & = -\epsilon \nabla \cdot (\eta \mathbf v) + \mathcal O (  \mu^2+\mu \epsilon),
  \\
  \partial_t \mathbf v + \nabla \eta +  \frac{\mu}{3} \Delta\nabla \eta   &=   - \epsilon \nabla ( |\mathbf v |^2/2) +  \mathcal O (  \mu^2+\mu \epsilon),
\end{aligned}
\right.
\end{equation*} 
which is a perturbation of the Boussinesq system
\begin{equation}\label{wtbq-approx}
\left\{
\begin{aligned}
\partial_t \eta +  \nabla \cdot  \mathbf v & = -\epsilon \nabla \cdot (\eta \mathbf v) ,
  \\
  \partial_t \mathbf v + \nabla \eta +  \frac{\mu}{3} \Delta\nabla \eta   &=   - \epsilon \nabla ( |\mathbf v |^2/2) .
\end{aligned}
\right.
\end{equation} 
The later system is a particular member of the (abcd) family of Boussinesq systems derived in \cite{BCS2002} as asymptotic models for water waves in the Boussinesq regime. Unfortunately, this system is linearlity ill-posed and thus cannot be used as a relevant water wave model. The system \eqref{wt2d} can be viewed as a regularization of this ill-posed system.

The system \eqref{wt2d} enjoys the Hamiltonian structure
\[
	\partial_t (\eta,  \mathbf  v)^T = \mathcal J_\mu \nabla \mathcal H_\mu(\eta ,  \mathbf v)
\]
with the skew-adjoint matrix
\[
	\mathcal J_\mu
	=
	\begin{pmatrix}
		0 & - K_\mu \partial_{x_1} & - K_\mu \partial_{x_2}
		\\
		- K_\mu \partial_{x_1} & 0 & 0
		\\
		- K_\mu \partial_{x_2} & 0 & 0
	\end{pmatrix}
	.
\]
This in particular
guarantees conservation of the energy functional
\begin{equation}
\label{Hamiltonian2}
	\mathcal E_\mu(\eta, \mathbf v)  = \frac 12 \int_{\R^2}
	\left(	
		\eta^2 + \left| K_\mu^{-\frac12} \mathbf v \right|^2
		+ \eta |\mathbf v |^2
	\right)
	dx
	.
\end{equation}

The one dimensional version of \eqref{wt2d} that
describes the evolution with time of surface waves of
a liquid layer in the two dimensional physical space
is written as 
\begin{equation}
\label{wt1d}
\left\{
\begin{aligned}
  \partial_t \eta +  \partial_x v & = -\epsilon  L_\mu \partial_x(\eta v),
  \\
  \partial_t v + L_\mu \partial_x \eta  &=   - \epsilon   L_\mu  \partial_x (  v ^2/2),
  \\
  (\eta, v)(0)&=( \eta_0, v_0) ,
\end{aligned}
\right.
\end{equation} 
where $\eta, v: \R^{2+1}\mapsto \R$ and 
\begin{equation*}
	L_\mu :=L_\mu (D)= \frac{\tanh( \sqrt \mu D)}{\sqrt \mu D}\quad \text{with}  \ D=-i\partial_x.
\end{equation*}
 For more details on the study of \eqref{wt2d}, \eqref{wt1d} or related equations, we refer the reader to \cite{Cr18, Dy19, DDK19, DST20, DDT22, T22}.

Recently,  together with Dinvay and Selberg \cite{DST20} we studied low regularity well-posedness of the Cauchy problems \eqref{wt2d} and \eqref{wt1d} for $\mu=\epsilon=1$. In particular,  we proved that \eqref{wt1d} is globally well-posed for initial data $( \eta_0, v_0)$ that is small in the $L^2(\R) \times H^{1/2} ( \R )$--norm. Moreover,  we showed that \eqref{wt2d} is locally well-posed for initial data $( \eta_0, \mathbf v_0) \in H^s(\R^2) \times H^{s+1/2} ( \R^2 ) \times H^{s+1/2} ( \R^2 )$ with $s>1/4$.

In the present paper, we are interested in the problem of long-time existence of solution to \eqref{wt2d} assuming that the nonlinearity parameter $\epsilon$ is sufficiently small.  In particular, by exploiting the dispersive nature of the system we prove that  \eqref{wt2d} is well-posed with existence time of scale $\mathcal O\left(  \mu^{3/2-} \epsilon^{-2+}\right)$ if $( \eta_0, \mathbf v_0) \in H^s(\R^2) \times H^{s+1/2} ( \R^2 ) \times H^{s+1/2} ( \R^2 )$ with $s>1/4$.
This in particular recovers the local well-posedness result in \cite[Theorem 2]{DST20} for $\mu=\epsilon= 1$.

There has been several studies by J-C. Saut \etal \ \cite{LPS14,  SX19, SX20, SX12,   SWX15}  (see also  \cite{ KLPS18, MP22, E21})
 regarding long-time existence of solutions for Whitham-Boussinesq type equations with initial data of size $\mathcal O\left( 1\right)$ in some Sobolev norm, where the time of existence depends on parameter $\epsilon$. In \cite{SX20, SX12, SWX15} the analysis is based only on symmetrization and energy techniques, and do not exploit the dispersive properties of the equations. The time of existence obtained for the equations involved is at most of scale $\mathcal O\left( 1/\epsilon \right)$ in the KdV regime $\{\mu \sim \epsilon $\}, but the space of resolutions are smaller. 
On the other hand, in \cite{LPS14} and \cite{SX19} the dispersive nature of the systems involved is used to study the long-time existence problem.
For instance, in \cite{LPS14} using the dispersive method the authors proved that the two dimensional dispersive Boussinesq system of the form
\begin{equation}
\label{bousq}
\left\{
\begin{aligned}
  \partial_t \eta +  \nabla \cdot (1+ \epsilon \Delta) \mathbf v & = - \epsilon  \nabla\cdot  (\eta \mathbf v),
  \\
  \partial_t \mathbf v + \nabla (1 +\epsilon \Delta) \eta  &=   -  \epsilon \nabla ( |\mathbf v |^2/2),
  \\
   (\eta, \mathbf v)(0)&=( \eta_0, \mathbf v_0) \in[ H^s(\R^2)]^3 
\end{aligned}
\right.
\end{equation} 
  is locally well-posed with existence time of scale $\mathcal O\left( 1/\sqrt \epsilon \right)$ whenever $s>3/2$.

Our main result is as follows.
\begin{theorem}
\label{theorem2d}
	Let $s > 1/4$ and $ \mu, \epsilon \in (0, 1] $. Suppose that $v_0$ is curl free, i.e., $\nabla \times \mathbf v_0 = 0$ and 
the initial data has size $$
	\left(	\lVert \eta_0 \rVert _{H^s(\R^2)}
		+ \lVert \mathbf v_0 \rVert _{ (H^{s + 1/2} (\R^2))^2 } \right) \sim \mathcal D_0.
	$$
	Then there is a solution
	$$
		( \eta, \mathbf v ) \in C \left([0, T];
		H^s(\mathbb{R}^2)
		\times
		\left( H^{s + 1/2} \left( \mathbb R^2 \right) \right) ^2
		\right)
$$
	of the Cauchy problem \eqref{wt2d}
	 with  existence time $T $ given by \footnote{We use the notation $a\pm:=a \pm \delta$ for sufficiently small $\delta>0$. For any positive numbers $A$ and $B$, the notation $A \lesssim B$ stands for $A \le  CB$, where $C$ is a positive constant that is independent $\mu$, $\epsilon$ or  $T$. Moreover, we denote $A \sim B$  when $A\lesssim B$ and $B \lesssim A$. }
	$$
	T \sim  \mathcal D_0^{-2+} \mu^{3/2-} \epsilon^{-2+}   .$$
	Moreover, the solution is unique in some subspace of the above solution space and the solution depends continuously on the initial data.
\end{theorem}

\begin{remark} From Theorem \ref{theorem2d} we deduce the following:
\begin{itemize}
\item In the regime $\{  \epsilon \ll \mu \sim 1\}$,  the solution exists on a larger time scale of order $\epsilon^{-2+}$. 
This is due to the presence of weak nonlinearities $\{ \epsilon \ll 1\}$, which is also regularized by the operator
$$K_\mu(D) \sim \angles{ \sqrt \mu D}^{-1} \sim \angles{  D}^{-1}  \quad \text{for} \ \ \mu\sim 1.$$

\item 
 In the regime $\{ \mu \ll  \epsilon \sim 1\}$, the time of existence is of order $\mu^{3/2-}$ and hence shrinks to $0$ as $\mu \rightarrow 0+$. 
This is due to the presence of strong nonlinearities $\{ \epsilon \sim 1\}$ which is not regularized, since  $$K_\mu(D) \sim  1-  \frac \mu3|D|^2  \quad \text{for} \ \ \mu\ll 1.$$
In fact, in this regime the system \eqref{wt2d} is a perturbation of the ill-posed system \eqref{wtbq-approx}.

\item 
Finally, in the KdV regime $\{\mu \sim \epsilon $\}, the existence time is of order $ \epsilon^{-1/2}$.
\end{itemize}

\end{remark}

In what follows we diagonalize \eqref{wt2d}, and then reduce Theorem \ref{theorem2d} to Theorem \ref{theorem2d-reduc} below which corresponds to the diagonalized system. To this end, we define
\begin{align*}
u_\pm=\frac{\eta \mp i K_\mu^{-1/2} \mathbf R\cdot \mathbf v}{2 },
\end{align*}
where $(\eta, \mathbf v)$ is a solution to \eqref{wt2d}, and 
 $\mathbf R=|D|^{-1}\nabla $ is the Riesz transform. Then we have
$$
\eta=u_++u_- \quad \mathbf v=-i \sqrt K_\mu \mathbf R (u_+- u_-).
$$
Set $$m_\mu(D):= |D|  \sqrt{K_\mu(D)}.$$
The system \eqref{wt2d} therefore transforms to
\begin{equation}
\label{wt2d-transf}
\left\{
\begin{aligned}
 (i\partial_t\mp m_\mu(D) ) u_\pm
&= \epsilon  \mathcal N^\pm_\mu(u_+, u_-) ,
\\
u_\pm (0)&= f_\pm,
\end{aligned}
\right.
\end{equation}
where
\begin{equation*}
 f_\pm =\frac{\eta_0 \mp i K_\mu^{-1/2} \mathbf  R\cdot \mathbf v_0}{2 } 
\end{equation*}
and 
\begin{equation}
\label{wtnonlin}
\begin{split}
\mathcal N^\pm_\mu(u_+, u_-)  &=   2^{-1}|D| K_\mu\mathbf R \cdot \left\{   (u_++u_-)    \mathbf R  \sqrt{K_\mu}(u_+ - u_-)\right\} 
\\
& \qquad \pm    4^{-1}|D|  \sqrt{K_\mu} \left|  \mathbf R \sqrt{K_\mu} (u_+ - u_-)\right|^2.
\end{split}
\end{equation}

Thus, Theorem \ref{theorem2d} redues to the following:
\begin{theorem}
\label{theorem2d-reduc}
	Let $s > 1/4$, $ \mu, \epsilon \in (0, 1] $, and that the initial data has size
	$$
		\sum_\pm \lVert f_\pm \rVert _{H^s}\sim\mathcal D_0.
	$$
	Then there exists a solution
	$$
		(u_+, u_-)\in C \left([0, T];
		H^s(\mathbb{R}^2)\times H^s(\mathbb{R}^2)
		\right)
$$
	of the Cauchy problem \eqref{wt2d-transf}--\eqref{wtnonlin} with existence time $T$ as in  Theorem \ref{theorem2d}.
	Moreover, the solution is unique in some subspace of 
	$ C \left([0, T];
		H^s(\mathbb{R}^2)\times H^s(\mathbb{R}^2)
		\right)$ and the solution depends continuously on the initial data.
\end{theorem}

\medskip    
The paper is organized as follows: In Section 2, we prove localized dispersive and Strichartz estimates for the linear propagators associated to \eqref{wt2d-transf}.
In Section 3 and 4 we prove Theorem \ref{theorem2d-reduc} and bilinear estimates that are crucial in the proof of Theorem \ref{theorem2d-reduc}.

\medskip
\noindent \textbf{Notation}. 
We fix an even smooth function $\chi \in C_0^{\infty}(\mathbb R)$ such that
\begin{equation*}
 0 \le \chi \le 1, \quad
\chi_{|_{[-1,1]}}=1 \quad \mbox{and} \quad  \mbox{supp}(\chi)
\subset [-2,2]
\end{equation*}
 and set
$$
\beta(s)
=\chi\left(s\right)-\chi \left(2s\right).
 $$
 For a dyadic number
  $\lambda \in  2^\Z$,  we set $\beta_{\lambda}(s):=\beta\left(s/\lambda\right)$, and thus $\supp \beta_\lambda= 
\{ s\in \R: \lambda/ 2 \le |s| \le 2\lambda \}$. 
Now define the frequency projection $P_\lambda$ via
\begin{align*}
\widehat{P_{\lambda} f}(\xi)  = \beta_\lambda(|\xi|)\widehat { f}(\xi) .
 \end{align*}
We sometimes write $f_\lambda:=P_\lambda f $, so that
\[ f=\sum_{\lambda  } f_\lambda ,\]
where summations throughout the paper are done over dyadic numbers in $ 2^\Z$.

The Sobolev space $H^s(\R^2)$ is defined via the norm 
$$
\norm{f}_{H^s} \sim \left[ \sum_{\lambda }  \angles{\lambda}^{2s}\norm{ f_\lambda}^2_{L^2_x} \right]^\frac 12, 
$$
 where $\angles{\xi}:= \left(1 +|\xi|^2\right)^{1/2}.$ If $B$ is a space of functions on $\R^2$, $T>0$ and $1\le p\le\infty$, we define the spaces $L^p\big((0,T) : B \big)$ and $L^p\big( \mathbb R : B\big)$ respectively via the norms
$$
\|f\|_{L^p_TB_x} = \left( \int_0^T \|f(\cdot,t)\|_{B}^p dt \right)^{\frac1p} \quad  \textrm{and} \quad \|f\|_{L^p_tB_x} = \left( \int_{\mathbb R} \|f(\cdot,t)\|_{B}^p dt \right)^{\frac1p} \, ,
$$
when $1 \le p < \infty$, with the usual modifications when $p=+\infty$.
\section{localised dispersive and Strichartz estimates}

First we derive lower bound estimates for the phase function
$$
m_\mu(r)=r \sqrt{K_\mu(r)}
$$
and its first two derivatives. Clearly,
$K_\mu(r) \sim \angles{ \sqrt \mu r}^{-1}$, 
and hence
\begin{equation}
\label{m-est}
|m_\mu(r)| \sim |r| \angles{ \sqrt \mu r}^{-1/2}.
\end{equation}

\begin{lemma}\label{lm-mest}
 For all $r>0$, we have
\begin{align}
\label{m'-est} 
0<m_\mu'(r)& \sim \angles{\sqrt \mu r}^{-1/2},
\\
\label{m''-est} 
0<-m_\mu ''(r)&\sim   \mu r \angles{\sqrt \mu r}^{-5/2}.
\end{align}
\end{lemma}

\begin{proof}
Since $$
 m_{\mu}(r)= \frac{1}{\sqrt{\mu}} m_{1}(\sqrt{\mu} r)$$
 \eqref{m'-est} and \eqref{m''-est} reduce to proving
\begin{align*}
0<m_{1}'(r)& \sim  \angles{r}^{-1/2},
\\
0< - m_{1}''(r)|&\sim    |r|  \angles{ r}^{-5/2}
\end{align*}
both of which are proved in \cite[Lemma 8]{DST20}.

\end{proof}

Van der Corput's Lemma will be useful in the derivation of the dispersive estimate in Lemma \ref{lm-dispest} below.
\begin{lemma}[Van der Corput's Lemma, \cite{S93} ]\label{lm-corput}
 Assume 
$g \in C^1(a, b)$, $\psi\in  C^2(a, b)$ and $|\psi''(r)|  \gtrsim A$ for all $r\in (a, b)$. Then 
\begin{align}
\label{corput} 
\Bigabs{\int_a^b e^{i t  \psi(r)}   g(r) \, dr}& \le C  (At)^{-1/2}  \left[ |g(b)| + \int_a^b |g'(r)| \, dr \right] ,
\end{align}
for some constant $C>0$ that is independent of $a$, $b$ and $t$.
\end{lemma}

Now we follow a similar argument\footnote{Inequality \eqref{dispest} is derived in \cite{DST20} in the case $\mu=1$ and $\lambda >1$.} as in \cite[Lemma 9]{DST20}
to derive 
a  frequency localized  $L^1-L^{\infty}$ decay estimate for the linear propagator associated to \eqref{wt2d-transf}:
$$
\mathcal S_{m_{\mu}}( \pm  t) := e^{\mp it m_{\mu}(D)}.
$$

\begin{lemma}[Localised dispersive estimate] \label{lm-dispest}
Let $0 < \mu \le 1$ and $\lambda\in 2^\Z$.
Then
\begin{align}
\label{dispest}
\| \mathcal S_{m_{\mu}}(\pm  t)P_{\lambda} f  \|_{L^\infty_x(\R^2)} &\lesssim  \min\left( \lambda^2, \ \mu^{-\frac12}  \angles{\sqrt \mu \lambda}^{\frac 32} t^{-1} \right)  \| f\|_{L_x^1(\R^2)} \end{align}
for all $f \in \mathcal{S}(\R^2)$.
 \end{lemma}
 \begin{proof}
 Without loss of generality we may assume $t>0$ and $\pm=-$. Now
we can write
\[
\left[ \mathcal S_{m_\mu}(- t) f_\lambda \right](x) 
= (I_{\lambda, \mu} (\cdot, t)\ast f)(x),
\]
where
\begin{equation}\label{Idef}
 I_{\lambda, \mu} (x, t)
=\lambda^2 \int_{\R^2} e^{i \lambda x \cdot \xi+ it {m_\mu}(\lambda \xi)}  \beta(\xi) \, d\xi .
\end{equation}
By Young's inequality
\begin{equation}\label{younginq}
\| S_{m_\mu}(-t) f_\lambda  \|_{L^\infty_x(\R^2)} \le \| I_{\lambda, \mu (\cdot, t)} \|_{L^\infty_x(\R^2)}  \|f\|_{L_x^1(\R^2)},
\end{equation}
and therefore, \eqref{dispest} reduces to proving
\begin{equation*}
\|I_{\lambda, \mu} (\cdot, t) \|_{L^\infty_x(\R^2)}  \lesssim  \min\left( \lambda^2, \ \mu^{-\frac12}  \angles{\sqrt \mu \lambda}^{\frac 32} t^{-1} \right) .
\end{equation*}
But clearly,
\begin{equation*}
\|I_{\lambda, \mu} (\cdot, t) \|_{L^\infty_x(\R^2)} 
\lesssim  \lambda^2,
\end{equation*}
and so we reduce further to proving 
\begin{equation}
\label{Iest-1}
\|I_{\lambda, \mu} (\cdot, t) \|_{L^\infty_x(\R^2)}  \lesssim   \mu^{-\frac12}  \angles{\sqrt \mu \lambda}^{\frac 32} t^{-1} .
\end{equation}

To prove \eqref{Iest-1} first observe that $I_{\lambda, \mu} (x, t)$ is radially symmetric w.r.t $x$, as it is 
 the inverse Fourier transform of the radial function $e^{it {m_\mu}(\lambda \xi)}  \beta(\xi) $, and so we may 
 set $x =(|x|, 0)$.
Therefore, using polar coordinates we can write
\begin{equation}\label{I-eq}
I_{\lambda, \mu} (x, t) =2\pi \lambda^2 \int_{1/2}^2  e^{it m_\mu(\lambda r) }J_0( \lambda r |x|) r\beta(r) \, dr,
\end{equation}
where $J_k(r)$ is the Bessel function:
$$
J_k(r)=\frac{ (r/2)^k}{(k+1/2) \sqrt{\pi}} \int_{-1}^1  e^{ir s} \left(1-s^2\right)^{k-1/2} \, ds \quad \text{for} \ k>-1/2.
$$
  The Bessel function $J_k(r)$ satisfies the following properties for $k>-1/2$ and $r>0$ (See   \cite[Appendix B]{G08} and \cite{S93}):
\begin{align}
\label{Jm1}
J_k (r) &\le Cr^{k} ,
\\
\label{Jm2}
J_k(r)& \le C r^{-1/2} ,
\\
\label{Jm3}
\partial_r \left[ r^{-k} J_k(r)\right] &= -r^{-k} J_{k+1}(r)
\end{align}
Moreover, we can write
\begin{equation}
\label{J0est}
J_0(s)= e^{is} h(s)  +e^{-is}\bar h(s)
\end{equation}
for some function $h$ satisfying the estimate 
\begin{equation}
\label{h-est}
| \partial_r ^j h(r)|\le C_j \angles{r}^{-1/2-j}  \quad \text{for all} \ j\ge 0.
\end{equation}

We prove \eqref{Iest-1} by treating the cases $  |x|\lesssim \lambda^{-1}$ and $  |x|\gg  \lambda^{-1}$ separately. 
\subsubsection{Case 1: $  |x|\lesssim  \lambda^{-1}$}
By \eqref{Jm1} and \eqref{Jm3} we have for all $ r\in (1/2, 2)$ the estimate
\begin{equation}
\label{J0derv-est}
\left| \partial_r^j J_0( \lambda r |x|) \right|\lesssim 1  \quad \text{for $j=0,1$}.
\end{equation}
We integrate by parts \eqref{I-eq} to obtain
\begin{align*}
 I_{\lambda, \mu} (x, t) 
&=-2\pi i\lambda t^{-1}  \int_{1/2}^2  \frac{d}{dr}\left\{ e^{it m_\mu (\lambda r) }\right\}  [ m_\mu'(\lambda r) ]^{-1} J_0( \lambda r |x|) \tilde\beta(r) \, dr
\\
&=2\pi i\lambda t^{-1}  \int_{1/2}^2   e^{it m_\mu (\lambda r) }[m_\mu'(\lambda r) ]^{-1} \partial_r\left[ J_0( \lambda r |x|) \tilde\beta(r) \right] \, dr
\\
 & \qquad -2\pi i\lambda t^{-1}  \int_{1/2}^2   e^{it m_\mu (\lambda r) } [m_\mu'(\lambda r) ]^{-2} \lambda m_\mu''(\lambda r) J_0( \lambda r |x|) \tilde\beta(r)  \, dr.
\end{align*}
Then applying Lemma \ref{lm-mest} 
and \eqref{J0derv-est} 
we obtain 
\begin{equation}
\label{I-est1-2d}
\begin{split}
|  I_{\lambda, \mu} (x, t) | &\lesssim \lambda  t^{-1} \left( \angles{ \sqrt \mu \lambda }^{\frac 12}+ \mu\lambda^2 \angles{ \sqrt \mu \lambda }^{-\frac32}  \right)
\\
&
\lesssim \lambda \angles{ \sqrt \mu \lambda }^{\frac12} t^{-1}  \lesssim \mu^{-\frac12}  \angles{\sqrt \mu \lambda}^{\frac 32} t^{-1} .
\end{split}
\end{equation}

\subsubsection{Case 2: $  |x|\gg  \lambda^{-1}$ }
Using \eqref{J0est} in \eqref{I-eq} we write
\begin{align*}
 I_{\lambda, \mu} (x, t) 
 &=2\pi \lambda^2 \left\{\int_{1/2}^2  e^{it \phi^+_{\lambda,\mu} (r)  }  h(\lambda r |x|)  \tilde \beta(r) \, dr +  \int_{1/2}^2  e^{-it \phi^-_{\lambda, \mu} (r)  } \bar  h(\lambda r |x|)  \tilde\beta(r) \, dr \right\},
\end{align*}
where 
$$
\phi^\pm_{\lambda, \mu} (r)=    \lambda r|x|/t  \pm  m_\mu (\lambda r) .
$$
Set $H_{\lambda}( |x|, r) =h(\lambda r |x|)  \tilde \beta(r)$. In view of \eqref{h-est} we have 
\begin{equation}
\label{Hest}
| H_{\lambda}( |x|, r)  | +  |\partial_r  H_{\lambda}( |x|, r)  | \lesssim   (\lambda |x|)^{-1/2}
\end{equation}
for all $ r\in (1/2, 2)$, 
where we also used the fact $\lambda |x|\gg 1$.

Now 
we write 
$$
 I_{\lambda, \mu} (x, t) 
= I^+_{\lambda, \mu} (x, t) 
+  I^-_{\lambda, \mu} (x, t) ,
$$
where
\begin{align*}
 I^+_{\lambda, \mu} (x, t) 
 &=2\pi \lambda^2 \int_{1/2}^2  e^{it \phi^+_{\lambda, \mu} (r)  } H_{\lambda}( |x|, r)  \, dr ,
 \\
I^-_{\lambda, \mu} (x, t) &= 2\pi \lambda^2
   \int_{1/2}^2  e^{-it \phi^-_{\lambda, \mu} (r)  }\bar H_{\lambda}( |x|, r)  \, dr .
\end{align*}

Observe that
$$
\partial_r \phi^\pm_{\lambda, \mu} (r)=  \lambda  \left[ |x|/t \pm  m_\mu'(\lambda r) \right],\qquad \partial_r^2\phi^\pm_{\lambda, \mu} (r)=     \pm  \lambda^2 m_\mu''(\lambda r),
$$
and hence by Lemma \ref{lm-mest}, 
\begin{equation}
\label{phi'+:est}
|\partial_r \phi^+_{\lambda, \mu} (r)|\gtrsim  \lambda\angles{ \sqrt \mu \lambda }^{-1/2}
\qquad
|\partial^2_r \phi^\pm_{\lambda, \mu} (r)| \sim  \mu  \lambda^3\angles{  \sqrt\mu \lambda }^{-5/2}
\end{equation}
for all $ r\in (1/2, 2)$, where we also used the fact that $m'$ is positive.

\subsubsection*{\underline{Estimate for  $I^+_{\lambda, \mu} (x, t)$ } }
By integration by parts we have
\begin{align*}
 I^+_{\lambda, \mu} (x, t)
&=2\pi i t^{-1} \lambda^2  \int_{1/2}^2  e^{it \phi^+_{\lambda, \mu} (r)  }  \left\{
\frac{\partial_r H_{\lambda}( |x|, r)  }{\partial_r \phi^+_{\lambda, \mu} (r) } -   \frac{\partial^2_r \phi^+_{\lambda, \mu} (r)   H_{\lambda}( |x|, r)  }{\left[\partial_r \phi^+_{\lambda, \mu} (r) \right]^{2}}\right\}\, dr.
\end{align*}
Then using \eqref{Hest} and \eqref{phi'+:est} we have
\begin{equation}
\label{I-est2-2d}
\begin{split}
| I^+_{\lambda, \mu} (x, t)|
&\lesssim  t^{-1} \lambda^2 \left( \lambda^{-1}\angles{ \sqrt \mu \lambda }^{\frac12}  + \mu \lambda \angles{ \sqrt \mu \lambda }^{-\frac 32}  \right)  (\lambda|x|)^{-\frac 12} 
\\
&\lesssim   \lambda \angles{ \sqrt \mu \lambda }^{\frac12}  t^{-1} \lesssim  \mu^{-\frac12}  \angles{\sqrt \mu \lambda}^{\frac 32} t^{-1},
\end{split}
\end{equation}
where we also used the assumption $  \lambda |x|\gg  1$.

\subsubsection*{\underline{Estimate for $I^-_{\lambda, \mu} (x, t)$}}
We treat the the non-stationary and stationary cases separately. In the non-stationary
case, where 
$$ |x | \ll   \angles{ \sqrt \mu \lambda }^{-1/2} t \quad \text{or} \quad |x| \gg   \angles{ \sqrt \mu \lambda } ^{-1/2} t, $$ we have 
$$
|\partial_r \phi^-_{\lambda, \mu} (r)|\gtrsim  \lambda\angles{ \sqrt \mu \lambda }^{-1/2},
$$
and hence $I^-_{\lambda, \mu} (x, t)$ can be estimated in exactly the same way as $I^+_{\lambda, \mu} (x, t)$ above, and satisfies the same bound.

So it remains to treat the stationary case: 
$$
|x | \sim  \angles{ \sqrt \mu \lambda }^{-1/2} t.
$$
 In this case,  
we use Lemma \ref{lm-corput}, \eqref{phi'+:est} and\eqref{Hest} to obtain 
\begin{equation}\label{Iest-station}
\begin{split}
| I^-_{\lambda, \mu} (x, t) 
&\lesssim \lambda^2 \left( \mu  \lambda^3\angles{  \sqrt\mu \lambda }^{-5/2}  t \right)^{-\frac12}\left[ |H_\lambda^- (x, 2) |+ \int_{1/2}^2 | \partial_r  H_\lambda^- (x, r)| \, dr\right]
\\
&\lesssim  \mu^{-\frac12}  \lambda^{\frac12 }  \angles{\sqrt \mu \lambda}^{\frac54}   t^{-\frac12} \cdot (\lambda |x|)^{-\frac 12}
\\
& \lesssim      \mu^{-\frac12}  \angles{\sqrt \mu \lambda}^{\frac 32} t^{-1}  ,
\end{split}
\end{equation}
where we also used the fact that $H_\lambda^- (x, 2) =0$.

 \end{proof}

\begin{lemma}[Localised Strichartz estimates]\label{lm-LocStr}
 Let $0 < \mu \le 1$ and $\lambda\in 2^\Z$.
 Assume that the pair $(q, r)$ is \emph{Strichartz admissible} in the sense that
  \begin{equation} \label{admissible}
 q> 2, \ r\ge 2 \quad \text{and} \quad \frac1r+\frac1q=\frac12 \, .
 \end{equation} 
 Then
 \begin{align}
\label{Strest1d}
\norm{ \mathcal S_{m_\mu}(\pm t) f_{\lambda}}_{ L^{q}_{t} L^{r}_{ x} (\R^{2+1}) } \lesssim \mu ^{-\frac1{2q}} \angles{\sqrt\mu \lambda }^{\frac3{2q}}  
\norm{  f_{\lambda}}_{ L^2_{ x}(\R^2 )} ,
\\
\label{Strest1d-inh}
\norm{ \int_0^t \mathcal S_{m_\mu}(\pm(t-s)) F_\lambda (s) \, ds}_{ L^{q}_{t} L^{r}_{ x} (\R^{2+1}) } \lesssim \mu ^{-\frac1{2q}} \angles{\sqrt\mu \lambda }^{\frac3{2q}}  
\norm{ F_{\lambda}}_{ L_t^1L^2_{ x}(\R^{2+1} )} 
\end{align}
for all  $f \in \mathcal{S}(\R^2)$ and  $F \in \mathcal{S}(\R^{2+1})$.
\end{lemma}
 \begin{proof}
  We shall use the Hardy-Littlewood-Sobolev inequality which
asserts that
\begin{equation}
\label{HLSineq}
\norm{|\cdot |^{-\gamma}\ast f}_{L^a(\R)} \lesssim \ \norm{ f}_{L^b(\R)} 
\end{equation}
whenever $1 < b < a < \infty$ and $0 < \gamma< 1$ obey the scaling condition
$$
\frac1b=\frac1a +1-\gamma.
$$

We prove only \eqref{Strest1d} since \eqref{Strest1d-inh}  follows from \eqref{Strest1d}  by the standard $TT^*$--argument.
First note that \eqref{Strest1d} holds true for the pair $(q, r)=(\infty, 2)$ 
as this is just the energy inequality.  So we may assume $q\in(2, \infty)$.

Let $q'$ and $r'$ be the conjugates of $q$ and $r$, respectively, i.e., $q'=\frac q{q-1}$ and  $r'=\frac r{r-1}$.
 By the standard $TT^*$--argument, \eqref{Strest1d} is equivalent to the estimate 
\begin{equation}
\label{TTstar}
\norm{ TT^\ast F }_{L^{q}_{t} L^{r}_{ x} (\R^{2+1}) } \lesssim \mu ^{-\frac1{q}} \angles{\sqrt\mu \lambda }^{\frac3{q}}  
\norm{ F  }_{ L^{q'}_{ t} L_x^{r'}(\R^{2+1} )},
\end{equation}
where 
\begin{equation}\label{TTastF}
\begin{split}
 TT^\ast F (x, t)&= \int_{\R}  \int_\R e^{i  x  \xi\mp i(t-s) {m_\mu}(  \xi)}  \beta^2_\lambda (\xi)   \widehat{F}( \xi, s)\, ds  d\xi
 \\
 &= \int_\R  K_{\lambda,  t-s} \ast F( \cdot,  s) \, ds,
 \end{split}
\end{equation}
with
\begin{align*}
 K_{\lambda ,t}(x)&= \int_{\R}  e^{i  x  \xi\mp it  {m_\mu}(  \xi)}  \beta^2_\lambda (\xi)   \, d\xi.
\end{align*}
Since $$K_{\lambda, t} \ast g (x)= \mathcal S_{m_\mu}(t)  P_\lambda g_\lambda (x)$$  
 it follows from \eqref{dispest} that
\begin{equation}\label{kest1}
\|K_{\lambda, t} \ast g \|_{L_x^{\infty}(\R)} \lesssim  \mu ^{-\frac12} \angles{\sqrt\mu \lambda }^{\frac32}  |t|^{-1}  \|g\|_{L_x^{1}(\R)}.
\end{equation}
On the other hand, we have by Plancherel 
\begin{equation}\label{kest2}
\|K_{\lambda, t} \ast g \|_{L_x^{2}(\R)} \lesssim    \|g\|_{L_x^{2}(\R)}.
\end{equation}
So interpolation between \eqref{kest1} and \eqref{kest2} yields
\begin{equation}\label{kest3}
\|K_{\lambda, t} \ast g \|_{L_x^{r}(\R)} \lesssim  \left[  \mu ^{-\frac12} \angles{\sqrt\mu \lambda }^{\frac32}\right]^{1- \frac2r }  |t|^{-\left(1-\frac 2r\right)} \|g\|_{L_x^{r'}(\R^2)}
\end{equation}
 for all $  r \in[2, \infty].$

Applying Minkowski's inequality to \eqref{TTastF}, and then  \eqref{kest3} and  \eqref{HLSineq}
with $(a, b)=(q , q' )$ and $\gamma= 1-2/r=2/q$,
 we obtain
 \begin{align*}
\norm{TT^\ast F }_{L^{q}_{t} L^{r}_{ x} (\R^{2+1})}
&\le \norm{   \int_\R \norm{ K_{\lambda, t-s,} \ast
   F(s, \cdot) }_{L_x^r (\R^2)}  \, ds}_{L^{q}_t(\R)}
  \\
 &\lesssim \mu ^{-\frac1{q}} \angles{\sqrt\mu \lambda }^{\frac3{q}}   \norm{  \int_\R  |t-s|^{-\frac2q }
  \norm{ F(s, \cdot) }_{ L_x^{r'}(\R^2)}  \, ds }_{L_t^{q}(\R)}
   \\
 &\lesssim  \mu ^{-\frac1{q}} \angles{\sqrt\mu \lambda }^{\frac3{q}}    \norm{  
  \norm{ F }_{L_x^{r'}(\R^2) }  }_{L^{q'}_{ t} (\R)}
    \\
 &=  \mu ^{-\frac1{q}} \angles{\sqrt\mu \lambda }^{\frac3{q}}  
  \norm{ F  }_{ L^{q'}_{ t} L_x^{r'}(\R^{2+1})} \, ,
\end{align*}
which is the desired estimate \eqref{TTstar}.
 \end{proof}

 \section{Proof of Theorem \ref{theorem2d-reduc}}
The bilinear terms in \eqref{wtnonlin} can be written as 
 \begin{equation*}
\begin{split}
\mathcal N^\pm_\mu(u_+, u_-)  &=   \frac12 \sum_{\pm_1, \pm_2} \pm_2 D K_\mu \mathbf R \cdot  \left( u_{\pm_1}  \mathbf R   \sqrt{K_\mu}u_{\pm_2} \right) 
\\
& \qquad \qquad \pm  \frac14 \sum_{\pm_1, \pm_2}  (\pm_1) (\pm_2) D \sqrt{K_\mu }  \left( \mathbf R   \sqrt{K_\mu }   u_{\pm_1}   \cdot \mathbf R  \sqrt{K_\mu} u_{\pm_2} \right) ,
\end{split}
\end{equation*}
where $\pm_1$ and $\pm_2$ are independent signs.
Then the Duhamel's representation of \eqref{wt2d-transf} is given by
\begin{equation}
\label{uinteq}
\begin{split}
  u_\pm(t) &= \mathcal S_{m_\mu}(\pm t) f_\pm \mp \frac{ i  \epsilon}2 \sum_{\pm_1, \pm_2}  (\pm_2)  \mathcal A^\pm_{\mu} (u_{\pm_1},  u_{\pm_2})(t) 
  \\
  & \qquad \qquad \qquad \mp \frac{ i  \epsilon}4   \sum_{\pm_1, \pm_2}   (\pm_1) (\pm_2) \mathcal B^\pm_{\mu} (u_{\pm_1},  u_{\pm_2})(t),
  \end{split}
   \end{equation}
 where 
 \begin{equation}
\label{AB} 
 \begin{split}
\mathcal A^\pm_{\mu} (u,v)(t)&:=\int_{0}^{t}    \mathcal S_{m_\mu}( \pm(t-t') D K_\mu  \mathbf R \cdot \left( u  \mathbf R \sqrt{K_\mu} v \right)(t') \,dt' ,
\\
\mathcal  B^\pm_{\mu}  (u,v)(t)&:= \int_{0}^{t} \mathcal S_{m_\mu}( \pm(t-t') D \sqrt{K_\mu} \left(   \mathbf R  \sqrt{K_\mu} u  \cdot   \mathbf R  \sqrt{K_\mu} v \right)(t') \,dt'  .
\end{split}
\end{equation}

Now let $(q, r)$ with $q, \, r >2$ be a Strichartz admissible pair.
We define the contraction space, $X^s_T$, via the norm
$$
\norm{u}_{X^s_T}= \left[ \sum_{\lambda}  \angles{\lambda}^{2s}\norm{ u}^2_{X_\lambda} \right]^\frac 12,
$$
where
$$
\norm{u}_{X_\lambda} =
  \left[ \norm{  P_\lambda u}^2_{ L_T^{\infty}L_x^{2} } +\mu ^{\frac1{q}} \angles{\sqrt\mu \lambda }^{-\frac3{q}}   \norm{ P_\lambda u}^2_{ L_T^{q} L_x^{r}  }  \right]^\frac12.
$$
Observe that
\begin{equation}
\label{Xlam-est}
\begin{split}
\norm{ P_\lambda  u}_{ L_T^{\infty}L_x^{2} }  &\le \norm{  u}_{X_\lambda} ,
\\
\norm{  P_\lambda u}_{L_T^{q} L_x^{r }  }  &\le  \mu ^{-\frac1{2q}} \angles{\sqrt\mu \lambda }^{\frac3{2q}}\norm{  u}_{X_\lambda} .
\end{split}
\end{equation}
Moreover,
$$
  X^s_T \subset L_T^\infty H^s.
$$

 We estimate the linear part of \eqref{uinteq} using 
 \eqref{Strest1d} as follows:
 \begin{equation}
\label{homest} 
\begin{split}
 \norm{\mathcal S_{m_\mu}(\pm t)  f_\pm}_{X^s_T} &= \left[ \sum_{\lambda }  \angles{ \lambda}^{2s} \norm{  \mathcal S_{m_\mu}(\pm t) f_\pm }^2_{ X_\lambda}  \right]^\frac 12
 \\
 &\lesssim  \left[ \sum_{\lambda} \angles{ \lambda}^{2s} \norm{  P_\lambda  f_\pm}^2_{ L_x^{2} }  \right]^\frac 12
 \\
 & \sim \norm{   f_\pm}_{ H^{s} } .
\end{split}
\end{equation} 
So Theorem \ref{theorem2d-reduc} reduces to proving the following bilinear estimates whose proof will be given in the next section.
\begin{lemma}\label{lm-keybiest}
Let $0<\alpha\ll 1$, $0< \mu\le 1$, $s\ge 1/4+ \alpha$ and $T>0$. Then for
all $u, \ v \in X^s_T$, we have
\begin{align}
\label{Keybiest1}
 \norm{ \mathcal A^\pm_{\mu} (u,v) }_{X^s_T} &\lesssim  T^{\frac12+\alpha}  \mu^{-\frac 34+ \frac\alpha 2} \norm{u}_{X^s_T}  \norm{v}_{X^s_T}  ,
 \\
 \label{Keybiest2}
 \norm{ \mathcal B^\pm_{\mu} (u,v)   }_{X^s_T} &\lesssim  T^{\frac12+\alpha}  \mu^{-\frac 34+ \frac\alpha 2} \norm{u}_{X^s_T}  \norm{v}_{X^s_T}  ,
\end{align}
where $\mathcal A^\pm_{\mu}$ and $\mathcal B^\pm_{\mu}$ are as in \eqref{AB}.

\end{lemma}

Indeed, given that Lemma \ref{lm-keybiest} holds,
we solve the integral equations \eqref{uinteq} by contraction mapping techniques
as follows. Define the mapping
\[
  (u_+, u_-) \mapsto \left (\Phi^+_\mu (u_+, u_-), \Phi_\mu^-(u_+, u_-)\right),
\]
where $\Phi^\pm_\mu(u_+, u_-)$ is given by the right hand side of \eqref{uinteq}. 

Now given initial data with norm 
$$\sum_\pm \| f_\pm\|_{ H^s}\le \mathcal D_0 , $$ 
we look for a solution in the set
\[
  E_T = \left\{  (u_\pm  \in X^s_T
  \colon  \sum_\pm \| u_\pm\|_{ X^s_T} \le 2 C \mathcal D_0 \right\}.
\]
Then by \eqref{homest} and Lemma \ref{lm-keybiest} we have
\begin{align*}
\sum_\pm  \norm{\Phi^\pm_\mu(u_+, u_-)}_{ X^s_T} &\le C\sum_\pm  \norm{   f_\pm}_{ H^{s} } + C T^{\frac12+\alpha}  \mu^{-\frac 34+ \frac\alpha 2} \epsilon \left(\sum_\pm \| u_\pm\|_{ X^s_T}  \right)^2
\\
&\le C \mathcal D_0 + C T^{\frac12+\alpha}  \mu^{-\frac 34 + \frac\alpha 2}  \epsilon ( 2C \mathcal D_0)^2
\\
&\le 2 C \mathcal D_0,
\end{align*}
provided that 
\begin{equation}\label{T}
T\le \left( 8 C^2 \mathcal D_0\right)^{-2+ \delta} \mu^{\frac32-\delta } \epsilon^{-2+\delta },
\end{equation}
where $\delta = \frac {4\alpha } {1+2\alpha} \ll 1$.

Similarly, for two pair of solutions $(u_+, u_-)$ and $(v_+, v_-)$ in $E_T$ with the same data, one can derive the difference estimate
\begin{align*}
 &\sum_\pm  \norm{ \Phi^\pm_\mu(u_+, u_-)- \Phi^\pm_\mu(v_+, v_-)}_{ X^s_T} 
 \\
 & \quad \le C T^{\frac12+\alpha}  \mu^{-\frac 34+ \frac\alpha 2}  \epsilon \left(\sum_\pm \| u_\pm\|_{ X^s_T}+ \| v_\pm\|_{ X^s_T}  \right) \left(\sum_\pm \| u_\pm- v_\pm\|_{ X^s_T}  \right)
  \\
 & \quad \le 4 C^2 T^{\frac12+\alpha}  \mu^{-\frac 34+ \frac\alpha 2}  \epsilon \mathcal D_0 \left(\sum_\pm \| u_\pm- v_\pm\|_{ X^s_T}  \right)
 \\
 & \quad \le \frac12   \left(\sum_\pm \| u_\pm- v_\pm\|_{ X^s_T}  \right),
\end{align*}
where in the last inequality we used \eqref{T}.

Therefore, $( \Phi^+_\mu, \Phi^-_\mu)$ is a contraction on $E_T$ and therefore it has a unique fixed point $(u_+, u_-) \in E_T$ solving the integral equation \eqref{uinteq}--\eqref{AB} on $\R^2\times [0, T]$, where $$T \sim  \mathcal D_0^{-2+}\mu^{\frac32-}  \epsilon^{-2+}.$$ Uniqueness in the space $X^s_T \times X^s_T$ and continuous dependence on the initial data can be shown in a similar way, by the difference estimates. This concludes the proof of Theorems \ref{theorem2d-reduc}.

\section{Proof of Lemma \ref{lm-keybiest}}

First we prove key bilinear estimates in Lemma \ref{lm-biest} below that will be crucial in the proof of Lemma  \ref{lm-keybiest}. 
By Bernstein inequality, we have for all $r_1, r_2 \in \R$ and $p\ge 1$,
\begin{equation}\label{Bern-est}
 \| |D|^{r_1} K^{r_2}_\mu \mathbf R f_{\lambda} \|_{L^p_x}   \lesssim  \lambda^{r_1} \angles{ \sqrt \mu \lambda}^{-r_2}  \|f_{\lambda}\|_{L^p_x},
\end{equation}
where $\mathbf R=|D|^{-1}\nabla $ is the Riesz transform.

For $\lambda_j \in2^\Z$ ($j=0,1, 2)$,
we denote
$$ \bm \lambda =(\lambda_0, \lambda_1, \lambda_2), \quad \lambda_1 \wedge \lambda_2=\min (\lambda_1, \lambda_2).
$$
\begin{lemma}
\label{lm-biest}
Let $0<\alpha\ll 1$, $0< \mu\le 1$ and $T>0$. Then for all $u\in X_{\lambda_1 }$ and $v \in X_{\lambda_2 }$ we have
\begin{align}
\label{keybiles11}
\norm{ |D| K_\mu   P_{\lambda_0} 
 \mathbf R \cdot \left(   u_{\lambda_1} \mathbf R \sqrt  K_\mu v_{\lambda_2} \right) }_{L_{T,x}^2}  
& \lesssim  C_{\mu, T}( \bm \lambda) 
\norm{u}_{  X_{\lambda_1 } }
\norm{v}_{  X_{\lambda_2 } },
\\
\label{keybiles12}
\norm{ |D| \sqrt {K_\mu   }
 P_{\lambda_0}\left(   \mathbf R  \sqrt  K_\mu  u_{\lambda_1}  \cdot \mathbf R \sqrt  K_\mu v_{\lambda_2} \right) }_{L_{T,x}^2}  
& \lesssim  \widetilde C_{\mu, T}(\bm \lambda) 
\norm{u}_{  X_{\lambda_1 } }
\norm{v}_{  X_{\lambda_2 } },
\end{align}
 where 
\begin{align}
\label{C}
C_{\mu, T}(\bm \lambda)  &=   T^{\alpha}  \mu^{-\frac14+\frac\alpha 2} \frac{\lambda_0 
(\lambda_1 \wedge \lambda_2)^{2\alpha} \angles{\sqrt\mu (\lambda_1 \wedge \lambda_2) }^{\frac34-\frac{3\alpha}2} }{ \angles{\sqrt{\mu} \lambda_0}  \angles{\sqrt{\mu} \lambda_2} ^{\frac12}}   ,
\\ 
\label{Ctilde}
\widetilde C_{\mu, T}(\bm \lambda)  &=  T^{\alpha}  \mu^{-\frac14+\frac\alpha 2} \frac{\lambda_0 (\lambda_1 \wedge \lambda_2)^{2\alpha}  \angles{\sqrt\mu (\lambda_1 \wedge \lambda_2) }^{\frac34-\frac{3\alpha}2}   }{\prod_{j=0}^2 \angles{\sqrt{\mu} \lambda_j} ^{\frac12}}. 
\end{align}
\end{lemma}
\begin{proof} 

We only prove \eqref{keybiles11} since the proof for \eqref{keybiles12} is similar.
By symmetry we may assume $\lambda_1 \le \lambda_2$. Let
$$
\frac1q=\frac12-\alpha, \quad \frac1r = \alpha
$$ 
so that $(q,r)$ is Strichartz admissible.
Then by H\"{o}lder, Bernstein inequality, \eqref{Bern-est} and \eqref{Xlam-est} we obtain
\begin{align*}
\text{LHS} \ \eqref{keybiles11} & \lesssim \lambda_0\angles{\sqrt{\mu} \lambda_0}^{-1} \norm{   \mathbf R \cdot( u_{\lambda_1} \mathbf R \sqrt  K_\mu v_{\lambda_2} ) }_{L_{T,x}^2}  
\\
&\lesssim   \lambda_0\angles{\sqrt{\mu} \lambda_0}^{-1}
 \angles{\sqrt{\mu} \lambda_2}^{-\frac12}  \cdot  T^{\frac12-\frac1q} \lambda_{1}^{\frac 2r
}  \norm{u_{\lambda_1}}_{ L_T^q L_x^r  } \norm{  v_{\lambda_2}}_{ L_T^\infty L_x^2 }  
\\
&\lesssim T^\alpha\mu^{-\frac14 + \frac\alpha 2}  \frac{\lambda_0\lambda_1^{2\alpha} \angles{\sqrt\mu \lambda_1}^{\frac34 -\frac{3\alpha} 2} }{ \angles{\sqrt{\mu} \lambda_0}  \angles{\sqrt{\mu} \lambda_2}^{\frac12} }   
   \norm{u}_{X_{\lambda_1} }
  \norm{ v }_{X_{\lambda_2}}
\end{align*}
which proves \eqref{keybiles11}.
\end{proof}

Now we are ready to prove Lemma  \ref{lm-keybiest}:  \eqref{Keybiest1} \&  \eqref{Keybiest2}. To this end, we decompose 
 $ u=\sum_{\lambda} u_{\lambda}$ and $v=\sum_{\lambda} v_{\lambda}.$
Note that by denoting
$$
a_{\lambda}:=  \norm{u }_{  X_{\lambda} }, \quad  b_{\lambda}:=  \norm{v}_{  X_{\lambda }} $$
we can write
\begin{equation}\label{Xs-redf}
\norm{u}_{  X^s_T }= \norm{ \left(  \angles{\lambda}^{s}  a_{\lambda } \right)}_{l^2_{\lambda}}, \quad
\norm{v}_{  X^s_T }= \norm{ \left(  \angles{\lambda}^{s}  b_{\lambda } \right)}_{l^2_{\lambda}}.
\end{equation}

We shall make a frequent use of of the following dyadic summation estimate, for $\mu, \lambda \in 2^\Z$ and $ c_1, c_2,  p>0$:
\begin{equation*}
\sum_{\mu \sim \lambda } a_\mu \sim a_\lambda, \qquad 
\sum_{c_1 \lesssim \lambda \lesssim c_2} \lambda^p \lesssim 
\begin{cases}
& c_2^p  \quad \text{if} \ p>0,
\\
& c_1^p  \quad \text{if} \ p<0.
\end{cases}
\end{equation*} 

\subsection{Proof of \eqref{Keybiest1} }
Using \eqref{Strest1d-inh} and H\"{o}lder, we have
  \begin{align*}
\norm{   \mathcal A^\pm_{\mu} (u,v)  }_{X^s_T} ^2 
& \lesssim \sum_{\lambda_0}  \angles{\lambda_0}^{2s} \norm{   |D| K_\mu    \mathbf R \cdot P_{\lambda_0}  \left( u  \mathbf R \sqrt{K_\mu} v \right) }^2_{ L_T^1L_x^{2} } 
\\
  & \lesssim T \sum_{\lambda_0}  \angles{\lambda_0}^{2s} \norm{   |D| K_\mu    \mathbf R \cdot P_{\lambda_0}  \left( u  \mathbf R \sqrt{K_\mu} v \right) }^2_{ L_{T, x}^{2} } .
 \end{align*}
 But the dyadic decomposition 
 \begin{equation}\label{Iloc}
 \norm{   |D| K_\mu    \mathbf R \cdot P_{\lambda_0}  \left( u  \mathbf R \sqrt{K_\mu} v \right) }_{ L_{T, x}^{2} } 
 \lesssim \sum_{\lambda_1, \lambda_2}   \norm{   |D| K_\mu    \mathbf R \cdot P_{\lambda_0}  \left( u_{\lambda_1}  \mathbf R \sqrt{K_\mu} v_{\lambda_2}  \right) }_{ L_{T, x}^{2} } .
 \end{equation}
Now let $\lambda_\text{min}  $, $\lambda_\text{med}$  and $ \lambda_{\max}$ denote the minimum, median and the maximum of $\{\lambda_0, \lambda_1, \lambda_2\}$, respectively.
 By checking the support properties in Fourier space of the bilinear term on the right hand side of \eqref{Iloc} one can see that this term vanishes unless $\bm \lambda=(\lambda_0, \lambda_1, \lambda_2) \in   \Lambda,$
where
\begin{equation*}
  \Lambda= \{\bm \lambda : \  \lambda_\text{med} \sim \lambda_{\max}  \}.
  \end{equation*} 
Thus, we have a non-trivial contribution in \eqref{Iloc} if $\bm \lambda \in \cup_{j=0}^2  \Lambda_j,$ where
\begin{align*}
\Lambda_0 &= \{\bm \lambda : \ \lambda_0 \lesssim \lambda_1\sim  \lambda_2 \}, 
\\ 
\Lambda_1&= \{\bm \lambda : \ \lambda_2 \ll  \lambda_1\sim  \lambda_0 \}, 
\\ 
\Lambda_2 &= \{\bm \lambda : \ \lambda_1 \ll \lambda_2 \sim  \lambda_0 \}.
\end{align*}
 By using these facts, and applying \eqref{keybiles11} to the right hand side of \eqref{Iloc}, we get 
 \begin{align*}
\norm{   \mathcal A^\pm_{\mu} (u,v)  }_{X^s_T} ^2 
& \lesssim T  \sum_{j=0}^2   \mathcal I^{(j)}_{\mu, T},
 \end{align*}
 where 
\begin{equation}
\label{Gammadefj}
  \mathcal I^{(j)}_{\mu, T}= \sum_{\lambda_0}  \angles{\lambda_0}^{2s} \left[  \sum_{\lambda_1, \lambda_2: \ \bm \lambda \in   \Lambda_j }    C_{\mu, T}( \bm \lambda)   a_{\lambda_1 } b_{\lambda_2 }    \right]^2
\end{equation} 
with $C_{\mu, T}(\bm \lambda)$ as in \eqref{C}.

 So \eqref{Keybiest1} reduces to proving 
\begin{equation}
\label{GammaEstj}
 \mathcal I^{(j)}_{\mu, T} \lesssim  T^{2\alpha}  \mu^{-\frac 32+ \alpha} \norm{u}^2_{  X^s_T }\norm{v}^2_{ X^s_T } \quad (j=0,1,2).
\end{equation} 
Further, observe that $\Lambda_j\subset \cup_{k=0}^2  \Lambda_{jk},$ where
\begin{align*}
\Lambda_{j0} &= \{\bm \lambda \in \Lambda_j : \  \lambda_{\max}  \lesssim \mu^{-1/2} \}, 
\\
\Lambda_{j1} &= \{\bm \lambda \in \Lambda_j : \  \lambda_{\min} \lesssim \mu^{-1/2} \ \& \
 \lambda_\text{med}  \gg \mu^{-1/2}\}, 
 \\ 
 \Lambda_{j2}& = \{\bm \lambda \in \Lambda_j : \  \lambda_{\min}  \gg \mu^{-1/2}\}.
\end{align*}
So 
$$ \mathcal I^{(j)}_{\mu, T} \lesssim \sum_{k=0}^2 \mathcal I^{(jk)}_{\mu, T}, $$
 where
\begin{equation}
\label{Gammadefjk}
  \mathcal I^{(jk)}_{\mu, T}= \sum_{\lambda_0}  \angles{\lambda_0}^{2s} \left[  \sum_{\lambda_1, \lambda_2: \ \bm \lambda \in   \Lambda_{jk} }    C_{\mu, T}( \bm \lambda)   a_{\lambda_1 } b_{\lambda_2 }    \right]^2.
\end{equation} 
Thus, \eqref{GammaEstj} reduces further to proving 
\begin{equation}
\label{GammaEst-jk}
\mathcal I^{(jk)}_{\mu, T}  \lesssim  T^{2\alpha}  \mu^{-\frac 32+ \alpha} \norm{u}^2_{  X^s_T }\norm{v}^2_{ X^s_T } \quad (j,k=0, 1,2).
\end{equation}

We prove \eqref{GammaEst-jk} as follows.
\subsubsection{Estimates for $\mathcal I^{(0k)}_{\mu, T} $ ($k=0,1,2$)  } 
In the case $\bm \lambda \in \Lambda_{00} $, we have $$C_{\mu, T}(\bm \lambda)  \lesssim T^{\alpha}  \mu^{-\frac1{4} + \frac \alpha 2} \lambda_0 \lambda_2^{2\alpha} , $$ and hence
\begin{align*}
 \mathcal I^{(00)}_{\mu, T} 
&
\lesssim  T^{2\alpha} 
\mu^{-\frac12+ \alpha}  \sum_{ \lambda_0 \lesssim \mu^{-1/2}}  \lambda_0^2 \left( \sum_{ \lambda_1\sim \lambda_2\lesssim  \mu^{-1/2}}  
\angles{\lambda_1}^{s}  a_{\lambda_1 }  \cdot \angles{\lambda_2}^{2\alpha} b_{\lambda_2 } \right)^2
\\
&
\lesssim   T^{2\alpha}  \mu^{-\frac 32+ \alpha} \norm{u}^2_{  X^s_T }\norm{v}^2_{ X^s_T }
\end{align*}
for all $s\ge 2\alpha$, 
where to obtain the second line we used Cauchy Schwarz inequality in $\lambda_1\sim \lambda_2$ and \eqref{Xs-redf}.

If $\bm \lambda \in \Lambda_{01} $, we have $$C_{\mu, T}(\bm \lambda)  \lesssim T^{\alpha}   \mu^{-\frac{1}{8}-\frac\alpha 4}  \lambda_0 \lambda_2^{\frac14+\frac\alpha 2}, $$ and hence
\begin{align*}
\mathcal I^{(01)}_{\mu, T} 
&
\lesssim
T^{2\alpha} 
\mu^{-\frac14- \frac \alpha 2}    \sum_{ \lambda_0 \lesssim \mu^{-1/2}}  \lambda_0^2\left( \sum_{ \lambda_1\sim \lambda_2\gg \mu^{-1/2}}    \angles{\lambda_1}^{s} 
a_{\lambda_1 }  \cdot \ \lambda_2^{\frac14+\frac\alpha 2} b_{\lambda_2 } \right)^2
\\
&
\lesssim  
T^{2\alpha} 
\mu^{-\frac54- \frac \alpha 2}   \norm{u}^2_{  X^s_T }\norm{v}^2_{ X^s_T }
\end{align*}
for all $s\ge 1/4+ \alpha /4$.

Finally, if $\bm \lambda \in \Lambda_{02} $, we have $$C_{\mu, T}(\bm \lambda)  \lesssim T^{\alpha}  \mu^{-\frac{5}{8}-\frac\alpha 4}   \lambda_2^{\frac14+\frac\alpha 2} ,$$ and hence
\begin{align*}
\mathcal I^{(02)}_{\mu, T} 
&
\lesssim 
T^{2\alpha} 
\mu^{-\frac54- \frac \alpha 2}  
  \sum_{ \lambda_0 \gg \mu^{-1/2}}  \angles{\lambda_0}^{\frac12+ \alpha -2s} \left( \sum_{ \lambda_1\sim \lambda_2\gg \mu^{-1/2}}    \angles{\lambda_1}^{s} 
a_{\lambda_1 }  \cdot \angles{\lambda_2}^{s}  b_{\lambda_2 } \right)^2
\\
&
\lesssim  T^{2\alpha} 
\mu^{-\frac54- \frac \alpha 2}   \norm{u}^2_{  X^s_T }\norm{v}^2_{ X^s_T }
\end{align*}
for all $s> 1/4+\alpha/ 2$.

\subsubsection{Estimates for $\mathcal I^{(1k)}_{\mu, T} $ ($k=0,1,2$) } \label{secGamma1k}
If $\bm \lambda \in \Lambda_{10} $, we have $$C_{\mu, T}(\bm \lambda)  \lesssim T^{\alpha}  \mu^{-\frac1{4} + \frac \alpha 2} \lambda_0 \lambda_2^{2\alpha} .$$
Hence
\begin{align*}
\mathcal I^{(10)}_{\mu, T} 
&
\lesssim T^{2\alpha} 
\mu^{-\frac12+ \alpha} \sum_{ \lambda_0 \lesssim \mu^{-1/2}}   \angles{\lambda_0}^{2s}\lambda_0^2 \left( \sum_{ \lambda_2 \ll \lambda_1 \sim \lambda_0\lesssim \mu^{-1/2} }  
a_{\lambda_1 } \cdot  \lambda_2^{2\alpha}  b_{\lambda_2 } \right)^2
\\
& \lesssim T^{2\alpha} 
\mu^{-\frac 32+ \alpha }  \sum_{ \lambda_0 \lesssim \mu^{-1/2}}   \angles{\lambda_0}^{2s} a_{\lambda_0}^2 \left( \sum_{ \lambda_2  \lesssim \mu^{-1/2} }  
  \lambda_2^{2\alpha}  b_{\lambda_2 } \right)^2
\\
&
\lesssim T^{2\alpha} 
\mu^{-\frac 32+ \alpha } \norm{u}^2_{  X^s_T }\norm{v}^2_{ X^s_T },
\end{align*}
where to get the last two inequalities we used 
$\sum_{\lambda_1 \sim \lambda_0} a_{\lambda_1} \sim a_{\lambda_0} $ and by Cauchy Schwarz
\begin{align*}
 \sum_{ \lambda_2  \lesssim \mu^{-1/2} }  
  \lambda_2^{2\alpha}  b_{\lambda_2 } & \le \left( \sum_{ \lambda_2  \le 1}  
  +  \sum_{ \lambda_2  > 1}  \right)
  \lambda_2^{2\alpha}  b_{\lambda_2 } \\
  &\lesssim  \norm{ \left(  b_{\lambda_2 } \right)}_{l^2_{\lambda_2}}+ \norm{ \left(  \angles{\lambda_2}^{s}  b_{\lambda_2 } \right)}_{l^2_{\lambda_2}} \lesssim \norm{v}_{ X^s_T }
\end{align*}
for all $s>2\alpha $.

If $\bm \lambda \in \Lambda_{11} $, we have $$C_{\mu, T}(\bm \lambda)  \lesssim T^{\alpha}  \mu^{-\frac3{4} + \frac \alpha 2}  \lambda_2^{2\alpha} ,  $$
and hence using the previous argument
\begin{align*}
\mathcal I^{(11)}_{\mu, T} 
&
\lesssim
T^{2\alpha} 
\mu^{-\frac 32+ \alpha }  \sum_{ \lambda_0 \gg \mu^{-1/2}}   \angles{\lambda_0}^{2s}  \left( \sum_{  \lambda_1 \sim \lambda_0\gg \lambda_2, \ \lambda_2 \lesssim \mu^{-1/2}  }  
a_{\lambda_1 }  \lambda_2^{2\alpha}    b_{\lambda_2 } \right)^2
\\
&
\lesssim  T^{2\alpha} 
\mu^{-\frac 32+ \alpha }  \norm{u}^2_{  X^s_T }\norm{v}^2_{ X^s_T }
\end{align*}
for all $s>2\alpha $.

Finally, if $\bm \lambda \in \Lambda_{12} $, we have $$C_{\mu, T}(\bm \lambda)  \lesssim T^{\alpha}  \mu^{-\frac 58 -  \frac {\alpha} 4} \lambda_2^{\frac14+ \frac\alpha2}  ,  $$
and hence using the previous argument
\begin{align*}
\mathcal I^{(12)}_{\mu, T} 
&
\lesssim
T^{2\alpha} 
\mu^{-\frac 54-\frac \alpha 2 }  \sum_{ \lambda_0 \gg \mu^{-1/2}}   \angles{\lambda_0}^{2s}  \left( \sum_{  \lambda_1 \sim \lambda_0\gg \lambda_2\gg  \mu^{-1/2}  }  
a_{\lambda_1 }  \lambda_2^{ \frac14 + \frac\alpha 2}    b_{\lambda_2 } \right)^2
\\
&
\lesssim  T^{2\alpha} 
\mu^{-\frac 32+ \alpha }\norm{u}^2_{  X^s_T }\norm{v}^2_{ X^s_T }
\end{align*}
for all $s> 1/4 + \alpha /2$.

\subsubsection{Estimates for $\mathcal I^{(2k)}_{\mu, T} $ ($k=0,1,2$)  } 
If $\bm \lambda \in \Lambda_{20} $, we have $$C_{\mu, T}(\bm \lambda)  \lesssim T^{\alpha}  \mu^{-\frac1{4} + \frac \alpha 2} \lambda_0 \lambda_1^{2\alpha}  $$
and hence
 arguing as in the preceding subsection we obtain
\begin{align*}
 \mathcal I^{(20)}_{\mu, T}
&
\lesssim T^{2\alpha} 
\mu^{-\frac12+ \alpha} \sum_{ \lambda_0 \lesssim \mu^{-1/2}}   \angles{\lambda_0}^{2s}\lambda_0^2 \left( \sum_{ \lambda_1 \ll \lambda_2 \sim \lambda_0\lesssim \mu^{-1/2} }  
 \lambda_1^{2\alpha}  a_{\lambda_1 }   b_{\lambda_2 } \right)^2
 \\
&
\lesssim T^{2\alpha} 
\mu^{-\frac 32+ \alpha }\norm{u}^2_{  X^s_T }\norm{v}^2_{ X^s_T }
\end{align*}
for all $s>2\alpha $.

Next if $\bm \lambda \in \Lambda_{21} $, then $$C_{\mu, T}(\bm \lambda)  \lesssim T^{\alpha}  \mu^{-1 + \frac \alpha 2} \lambda_0^{-\frac12}  \lambda_1^{2\alpha}  , $$
and hence
\begin{align*}
 \mathcal I^{(21)}_{\mu, T}
&
\lesssim T^{2\alpha} 
\mu^{-2+ \alpha }  \sum_{ \lambda_0 \gg \mu^{-1/2}}   \angles{\lambda_0}^{2s}    \lambda_0^{-1} \left( \sum_{ \lambda_2 \sim \lambda_0\gg \lambda_1, \ \lambda_1 \lesssim \mu^{-1/2}  }  
 \lambda_1^{2\alpha}  a_{\lambda_1 }    b_{\lambda_2 } \right)^2
\\
&
\lesssim  T^{2\alpha} 
\mu^{-\frac 32+ \alpha }  \norm{u}^2_{  X^s_T }\norm{v}^2_{ X^s_T }
\end{align*}
for all $s>2\alpha $.

Finally, if $\bm \lambda \in \Lambda_{21} $, then $$C_{\mu, T}(\bm \lambda)  \lesssim T^{\alpha}  \mu^{-\frac 58 -  \frac {\alpha} 4}  \lambda_1^{\frac14+ \frac\alpha2}  , $$
and hence
\begin{align*}
  \mathcal I^{(22)}_{\mu, T}
&
\lesssim
T^{2\alpha} 
\mu^{-\frac 54-\frac \alpha 2 }  \sum_{ \lambda_0 \gg \mu^{-1/2}}   \angles{\lambda_0}^{2s}  \left( \sum_{ \lambda_2 \sim \lambda_0\gg \lambda_1\gg  \mu^{-1/2}  }   \lambda_1^{ \frac14 + \frac\alpha 2}   
a_{\lambda_1 }   b_{\lambda_2 } \right)^2
\\
&
\lesssim T^{2\alpha} 
\mu^{-\frac 54-\frac \alpha 2 } \norm{u}^2_{  X^s_T }\norm{v}^2_{ X^s_T }
\end{align*}
for all $s> 1/4 + \alpha / 2$.

\subsection{Proof of \eqref{Keybiest2} }
Arguing as in the preseding subsection we use \eqref{Strest1d-inh},  H\"{o}lder,  \eqref{keybiles12} and \eqref{Ctilde} to obtain
 \begin{align*}
\norm{   \mathcal B^\pm_{\mu} (u,v)  }_{X^s_T} ^2 
& \lesssim T \sum_{j=0}^2   \mathcal J^{(j)}_{\mu, T},
 \end{align*}
 where 
\begin{equation}\label{Gammatilde}
 \mathcal J^{(j)}_{\mu, T}= \sum_{\lambda_0}  \angles{\lambda_0}^{2s} \left[  \sum_{\lambda_1, \lambda_2 \in \Lambda_j}    \tilde C_{\mu, T}(\bm\lambda)   a_{\lambda_1 } b_{\lambda_2 }    \right]^2,
 \end{equation}
 with $\widetilde C_{\mu, T}(\bm \lambda)$ as in \eqref{Ctilde}.
 
 By symmetry it suffices to estimate $\mathcal J^{(0)}_{\mu, T}$ and $\mathcal J^{(1)}_{\mu, T}$.
Now if $\bm \lambda \in \Lambda_{0} $ or  $\bm \lambda \in \Lambda_{1} $, then from \eqref{C}--\eqref{Ctilde} we have
$$
 \widetilde C_{\mu, T}(\bm\lambda) \lesssim   C_{\mu, T}(\bm \lambda),
$$
and hence 
$$
 \mathcal J^{(j)}_{\mu, T} \lesssim  \mathcal I^{(j)}_{\mu, T}\lesssim T^{2\alpha}  \mu^{-\frac 32+ \alpha}  \norm{u}^2_{  X^s_T }\norm{v}^2_{ X^s_T } \quad (j=0,1),
$$
where to get the second inequality we used \eqref{GammaEstj}.

\vspace{8mm}

\noindent \textbf{Acknowledgments}
The author would like to thank the anonymous referee for useful comments on an earlier version of this article.


\end{document}